\documentclass{article}
\usepackage{latexsym}
\usepackage{amssymb}
\usepackage{amsmath}
\usepackage{color}
\usepackage{graphicx}
\usepackage{amsthm}
\usepackage{indentfirst}
\usepackage{mathrsfs}
\allowdisplaybreaks

\newtheorem{theorem}{\color{black}\indent Theorem}[section]
\newtheorem{lemma}{\color{black}\indent Lemma}[section]

\newtheorem{definition}{\color{black}\indent Definition}[section]
\newtheorem{remark}{\color{black}\indent Remark}[section]
\newtheorem{corollary}{\color{black}\indent Corollary}[section]

\renewcommand{\baselinestretch}{1.2}

\DeclareMathOperator{\diag}{{diag}}
\DeclareMathOperator{\ind}{{ind}}
\DeclareMathOperator{\dist}{{dist}}
\DeclareMathOperator{\fix}{{Fix}}
\begin{document}
\large
\title{Rotating Quasi-periodic Solutions of Second Order Hamiltonian Systems with Sub-quadratic Potential}
\author{{Jiamin Xing$^{a}$ \thanks{ E-mail address : xingjiamin1028@126.com.} ,~ Xue Yang$^{a,b}$ \thanks{ E-mail address : yangxuemath@163.com.} ,~ Yong Li$^{b,a}$} \thanks{ E-mail address : liyongmath@163.com}~\thanks{Corresponding author}\\
{$^{a}$School of Mathematics and Statistics, and}\\
{Center for Mathematics and Interdisciplinary Sciences,}\\
{Northeast Normal University, Changchun 130024, P. R. China.}\\
 {$^{b}$College of Mathematics, Jilin University,}\\
 {Changchun, 130012, P. R. China.}\\
 }
\date{}
\maketitle
\par
{\bf Abstract.}
\par This paper concerns the existence of multiple rotating quasi-periodic solutions for second order Hamiltonian systems with sub-quadratic potential. Such solutions have the form $x(t+T)=Qx(t)$ for some orthogonal matrix $Q$. To deal with such quasi-periodic solutions, we introduce the $\mathcal{Q}(s)$ index which is a development of the well known
$S^1$ index. Applying the $\mathcal{Q}(s)$ index, we give an estimate of the number for rotating quasi-periodic orbits with a fixed period.
\section{Introduction and main results}

Consider the second order Hamiltonian system
\begin{eqnarray}\label{1.1}
x''+\nabla V(x)=0,
\end{eqnarray}
where $V\in C^1(\mathbf{R}^{n},\mathbf{R})$. The existence theory of periodic solutions for system \eqref{1.1} has been well developed, for example, see \cite{Benci,Long,Ra,Ra1,Tang,Z-L,ZhangS} and the references therein. For quasi-periodic solutions, it is generally very difficult due to the small divisor. The celebrated KAM theory answers that most quasi-periodic solutions are persistent under small perturbations, see \cite{Kolmogorov,Arnold,Moser1} for non-resonant case and \cite{Li} for resonant case.
In the present paper, we study the existence of rotating quasi-periodic solutions with the form $x(t+T)=Qx(t)$ for some orthogonal matrix $Q$.
It is a symmetric periodic or quasi-periodic one. Recently, such solutions have been studied by many works.
Hu and Wang \cite{Hu1} established the theory of
conditional Fredholm determinant to study rotating quasi-periodic orbits in Hamiltonian systems. Hu et al \cite{Hu2} gave some stability criteria for the symmetric periodic orbits and used them to study the linear stability of elliptic Lagrangian solutions of the classical planar three-body problem. Chang and Li \cite{Chang} considered rotating quasi-periodic solutions of second order dissipative dynamical systems. Liu et al \cite{LiuG}
studied multiple  rotating quasi-periodic solutions of asymptotically linear Hamiltonian systems. Liu \cite{LiuH} and Zhang \cite{Zhang} obtained symmetric periodic orbits of Hamiltonian systems on a given convex energy surface respectively.

By the structure of \eqref{1.1}, one should consider rotating quasi-periodic solutions if $V$ is $Q$ invariant for some orthogonal matrix $Q$, that is $V(Qx)=V(x)$ for $x\in\mathbf{R}^{n}$. Clearly, the rotating quasi-periodic solution is periodic if $Q=I$ (identity matrix), anti-periodic if $Q=-I$, subharmonic if $Q^k=I$ for some positive integer $k$, or quasi-periodic if $Q^k\neq I$ for all $k\in\mathbf{N}$.

Throughout the paper we assume $V$ satisfies the following:
\begin{enumerate}
\item [(V1)] $V$ is twice differentiable at $x=0$, $V(0)=0$ and $\nabla V(0)=0$;
\item [(V2)] if $x\in\ker(I-Q)$ is a critical point of $V$, i.e.  $\nabla V(x)=0$, then $V(x)\leq 0$;
\item [(V3)] $V$ satisfies the $Q$ invariance mentioned above.
\end{enumerate}
For a solution $x(t)$ of \eqref{1.1}, the corresponding orbit is the set $x(\mathbf{R})$ and two solutions $x(t)$
and $y(t)$ are geometrically distinct or have different orbits if $x(\mathbf{R})\neq y(\mathbf{R})$.

Let
$$X=\{x\in H^1([0,T],\mathbf{R}^{n}): x(t+T)=Qx(t)\ {\rm for\ all}\ t\in\mathbf{R}\}.$$
Then $X$ is  a Hilbert space with the inner product
$$\langle x,y\rangle=\int_0^T\left((x(t),y(t))+(x'(t),y'(t))\right)\mathrm{d}t,$$
where $(\cdot,\cdot)$ denotes the inner product in $\mathbf{R}^{n}$. Let $|\cdot|$ denote the $2$ norm on $\mathbf{R}^n$ and
$\|\cdot\|$ the norm on $X$.
The linearized operator corresponding to \eqref{1.1} on $X$ is given by
$$Lx=x''+V_{xx}(0)x.$$
It follows from (V3) that
$$V_{xx}(0)Q=QV_{xx}(0).$$
Thus there exists a unitary matrix $P$ such that
\begin{align}\label{1.2}
P^{-1}V_{xx}(0)P=\diag\{\mu_1,\cdots,\mu_n\},
\end{align}
\begin{align}\label{1.3}
P^{-1}QP=\diag\{e^{i\theta_1},\cdots,e^{i\theta_n}\},
\end{align}
where $0\leq\theta_j<2\pi$ for $1\leq j\leq n$.
With a simple calculation, we obtain that the eigenvalues of $L$ on $X$ are
$$\lambda_m^j=\mu_j-\bigg(\frac{\theta_j+2\pi m}{T}\bigg)^2$$
for $m\in\mathbf{Z}$ and $1\leq j\leq n$. Let
$$p_T=\sharp\{(m,j): \lambda_m^j>0,\ \theta_j+2\pi m\neq0,\ 1\leq j\leq n,\ m\in\mathbf{Z}\}.$$

Now we state our main results.
\begin{theorem}\label{theorem1}
Assume $p_T>0$ and $V$ satisfies (V1), (V2), (V3) and
\begin{enumerate}
\item [(V4)] $\nabla V$ is bounded;
\item [(V5)] $V(x)\rightarrow+\infty$ for $x\in \ker(I-Q)$ and $|x|\rightarrow\infty$.
\end{enumerate}
Then system \eqref{1.1} has at least $\frac{p_T}{2}$ geometrically distinct $(Q,T)$-rotating quasi-periodic solutions.
\end{theorem}

\begin{remark}
By \eqref{1.2} and \eqref{1.3}, for $P=(\xi_1,\cdots,\xi_n)$, we have $Q\xi_j=e^{i\theta_j}\xi_j$ and
$V_{xx}(0)\xi_j=\mu_j\xi_j$. If $\theta_j\neq0$ or $\pi$, $\xi_j$ is a complex vector and one has
$$Q\bar{\xi}_j=e^{-i\theta_j}\bar{\xi}_j,\ \ \ \ V_{xx}(0)\bar{\xi}_j=\mu_j\bar{\xi}_j.$$
Clearly,  the multiplicity of eigenvalue $\mu_j$ is even. Then there exists a $1\leq p\leq n$ with $p\neq j$ such
that $\theta_p=2\pi-\theta_j$ and $\mu_p=\mu_j$. It follows that $\lambda_m^p=\lambda_{-m-1}^j$, and the number
$\sharp\{(m,r): \lambda_m^r>0,\ \theta_r+2\pi m\neq0,\ r=p, j\}$
is even. When $\theta_j=0$ or $\pi$,  $\sharp\{(m,j): \lambda_m^j>0,\ \theta_j+2\pi m\neq0\}$ is obviously even.
Hence $p_T$ is even and $\frac{p_T}{2}$ is an integer.
\end{remark}

We also have the following.
\begin{theorem}\label{theorem2}
Assume $p_T>0$ and $V$ satisfies (V1), (V2), (V3) and
\begin{enumerate}
\item [(V6)] there are $\beta\in(1,2)$ and $R>0$ such that
$$0<(x,\nabla V(x))\leq\beta V(x)\ \ {\rm for}\ \ |x|>R;$$
\item [(V7)] there are constants $a_1,\ a_2>0$ and $\alpha\in(1,\beta)$ such that $V(x)\geq a_1|x^\bot|^\alpha-a_2$ for all $x\in\mathbf{R}^{n}$, where $x^\bot$ denotes the projection of $x$ on $\ker(I-Q)$.
\end{enumerate}
Then system \eqref{1.1} has at least $\frac{p_T}{2}$ geometrically distinct $(Q,T)$-rotating quasi-periodic solutions.
\end{theorem}

\begin{remark}
In Theorem \ref{theorem1} and Theorem \ref{theorem2}, $V$ is sub-quadratic in the sense that it grows less than $|x|^2$. When $Q=I$ our results are consistent with Theorem 4.1 and Theorem 4.2 of Benci \cite{Benci}.
\end{remark}

When $\det(I-Q)\neq0$, assumption (V2) is contained in (V1) and (V7) contained in (V6). Hence we have the following.
\begin{corollary}\label{co1}
Assume $\det(I-Q)\neq0$, $p_T>0$  and $V$ satisfies (V1), (V3) and (V6).
Then system \eqref{1.1} has at least $\frac{p_T}{2}$ geometrically distinct $(Q,T)$-rotating quasi-periodic solutions.
\end{corollary}

\begin{remark}
When $Q=I$, coercive assumptions (V5) or (V7) are required to obtain the existence of periodic solutions. Corollary \ref{co1} indicates that the coercive condition can be replaced by some invariance, such as $V(x)=V(-x)$.
\end{remark}
To prove Theorem \ref{theorem1} and Theorem \ref{theorem2}, we introduce the $\mathcal{Q}(s)$ index and apply the Ljusternik-Schnirelmann theory of critical points. The $\mathcal{Q}(s)$ index is a development of $S^1$ index which is a powerful tool to study periodic solutions of Hamiltonian system. For literature, see \cite{Ekeland,Moser}.

The paper is organized as follows. In section 2 we introduce the definition of $\mathcal{Q}(s)$ index and show the properties that will be used in the proof of main results. In sections 3 and 4 we give the proof of Theorem \ref{theorem1} and Theorem \ref{theorem2} respectively following the method from Benci \cite{Benci} on periodic solutions.

\section{$\mathcal{Q}(s)$ index}
In this section, we introduce a new index to rotating quasi-periodic orbits. First, we recall the concept of index due to Rabinowitz \cite{ra-in}.

Suppose that $\mathcal{E}$ is a Banach space with a group $\mathfrak{g}$ acting on it. Set
$$\fix\mathfrak{g}=\{z\in\mathcal{E}:gz=z,\ {\rm for\ all}\ g\in\mathfrak{g}\}.$$
Let
$$\mathfrak{\Gamma}=\{\Gamma\subset\mathcal{E}\setminus\{0\}: g(\Gamma)\subset\Gamma\ {\rm for\ all}\ g\in\mathfrak{g}\}$$
be the set of $\mathfrak{g}$ invariant subsets of $\mathcal{E}\setminus\{0\}$.
\begin{definition}\label{def1}
An index for $(\mathcal{E},\mathfrak{g})$ is a mapping $i:\mathfrak{\Gamma}\rightarrow \mathbf{N}\cup\{\infty\}$ such that for
all $\Gamma_1,\Gamma_2\in\mathfrak{\Gamma}$,
\begin{enumerate}
\item[{\rm(i)}] Normalization: if $z\notin\fix\mathfrak{g}$, $i(\cup_{g\in\mathfrak{g}}gz)=1$;

\item[{\rm(ii)}] Mapping property: if $R: \Gamma_1\rightarrow\Gamma_2$ is continuous and equivariant which means
$Rg=gR$ for all $g\in\mathfrak{g}$, then
$$i(\Gamma_1)\leq i(\Gamma_2);$$

\item[{\rm(iii)}] Monotonicity property: If $\Gamma_1\subset\Gamma_2$, then $i(\Gamma_1)\leq i(\Gamma_2)$;

\item[{\rm(iv)}] Continuity property:  if $\Gamma_1$ is compact and $\Gamma_1\cap\fix\mathfrak{g}=\emptyset$, then
$i(\Gamma_1)<\infty$ and there exists a neighborhood $D\in\mathfrak{\Gamma}$ of $\Gamma_1$ such that
$$i(D)=i(\Gamma_1);$$

\item[{\rm(v)}]  Subadditivity: $i(\Gamma_1\cup\Gamma_2)\leq i(\Gamma_1)+i(\Gamma_2)$.
\end{enumerate}
\end{definition}

For $s\in\mathbf{R}$, consider the group action on $X$:
$$\mathcal{Q}(s)x(t)=x(t+s).$$
Clearly, $X$ is $\mathcal{Q}(s)$ invariant, that is, if $x\in X$, then $\mathcal{Q}(s)x\in X$ for all $s\in\mathbf{R}$.

Now we give the $\mathcal{Q}(s)$ index:
\begin{definition}\label{def2}
A $\mathcal{Q}(s)$ index $\ind \Gamma$ of an invariant subset $\Gamma$ of $X$ is the smallest integer $k$ such that there exists a
$$\Phi=(\Phi_1^\top,\cdots,\Phi_m^\top)^\top\in C(\Gamma,\mathbf{C}^{k}\setminus\{0\})$$
with $\Phi_j\in C(\Gamma,\mathbf{C}^{k_j})$ and
\begin{eqnarray}\label{2.4}
\Phi_j(\mathcal{Q}(s)x)=e^{i\frac{2\pi n_j+\theta_{p_j}}{T}s}\Phi_j(x)\ \ \forall s\in\mathbf{R},\  x\in\Gamma,
\end{eqnarray}
where $k_j\in\mathbf{N}_+$, $k_1+\cdots+k_m=k$, $n_j\in\mathbf{Z}$, $1\leq p_j\leq n$ and $2\pi n_j+\theta_{p_j}\neq 0$ for $1\leq j\leq m$.
If such a $\Phi$ does not exist, define $\ind \Gamma=\infty$, and if $\Gamma=\emptyset$, define $\ind \Gamma=0$.
\end{definition}

We need to show:
\begin{lemma}\label{p1}
The $\mathcal{Q}(s)$ index is one in the sense of Definition \ref{def1}.
\end{lemma}
\begin{proof}
{\rm (i)}  For $x\in X\setminus\fix\{\mathcal{\mathcal{Q}}(s)\}$, consider the function
\begin{eqnarray}\label{02.5}
y(t)=(y_1(t),\cdots,y_n(t))^\top=P^{-1}x(t).
\end{eqnarray}
Then one has
\begin{eqnarray}\label{2.5}
y(t+T)=P^{-1}x(t+T)=P^{-1}QPP^{-1}x(t)=(e^{i\theta_1}y_1(t),\cdots,e^{i\theta_n}y_n(t))^\top.
\end{eqnarray}
Hence for $1\leq j\leq n$,
$$e^{-i\frac{\theta_{j}}{T}(t+T)}y_j(t+T)=e^{-i\frac{\theta_{j}}{T}t}y_j(t)\ \ t\in\mathbf{R}.$$
Thus there exist a component $y_{j_0}(t)$ and $n_{j_0}\in\mathbf{Z}$ with $2\pi n_{j_0}-\theta_{{j_0}}\neq0$, such that
\begin{eqnarray}
\int_0^Te^{i\frac{2\pi n_{j_0}-\theta_{j_0}}{T}s}y_{j_0}(s)\mathrm{d}s\neq0.
\end{eqnarray}
Now for $\tilde{x}\in\Gamma_0=\{\mathcal{Q}(s)x: s\in\mathbf{R}\}$, assume $\tilde{x}(t)=x(t+s_0)$ for some $s_0\in \mathbf{R}$, and
$\tilde{y}(t)=P^{-1}\tilde{x}(t)$.
Let
$$\Phi(\tilde{x})=\int_0^Te^{i\frac{2\pi n_{j_0}-\theta_{j_0}}{T}s}\tilde{y}_{j_0}(s)\mathrm{d}s.$$
Then $\Phi\in C(\Gamma_0,\mathbf{C}\setminus\{0\})$,
and
$$\Phi(\mathcal{Q}(s)\tilde{x})=e^{i\frac{-2\pi n_{j_0}+\theta_{j_0}}{T}s}\Phi(\tilde{x}),$$
proving (i).

{\rm (ii)} If $\ind\Gamma_2=\infty$, the result is obvious. If $\ind\Gamma_2=k<\infty$, there exists a
$$\Phi=(\Phi_1^\top,\cdots,\Phi_m^\top)^\top\in C(\Gamma_2,\mathbf{C}^{k}\setminus\{0\})$$
such that $\Phi_j\in C(\Gamma_2,\mathbf{C}^{k_j})$ and
\begin{eqnarray}\label{3.3}
\Phi_j(\mathcal{Q}(s)x)=e^{i\frac{2\pi n_j+\theta_{p_j}}{T}s}\Phi_j(x)\ \ \forall s\in\mathbf{R},\  x\in\Gamma_2.
\end{eqnarray}
Define $\tilde{\Phi}(x)=\Phi\circ R(x)$ for $x\in\Gamma_1$. Then
\begin{align*}
\tilde{\Phi}_j(\mathcal{Q}(s)x)&=\Phi_j\circ R(\mathcal{Q}(s)x)=\Phi_j(\mathcal{Q}(s)R(x))\\
&=e^{i\frac{2\pi n_j+\theta_{p_j}}{T}s}\Phi_j\circ R(x)\\
&=e^{i\frac{2\pi n_j+\theta_{p_j}}{T}s}\tilde{\Phi}(x),
\end{align*}
which yields $\ind\Gamma_1\leq\ind\Gamma_2$.

{\rm (iii)} Since the inclusion map is equivariant, the  monotonicity property is obvious.

{\rm (iv)} When $\Gamma_1=\emptyset$, the result is obvious. When $\Gamma_1\neq\emptyset$, for each $x_0\in\Gamma_1$ and $y_0=P^{-1}x_0$, as in the proof of (i)
there exist a component $y_{j_0}(t)$ and $n_{0}\in\mathbf{Z}$ with $2\pi n_{0}-\theta_{{j_0}}\neq0$, such that
\begin{eqnarray}
\int_0^Te^{i\frac{2\pi n_{0}-\theta_{j_0}}{T}s}y_{j_0}(s)\mathrm{d}s\neq0.
\end{eqnarray}
Clearly, there exists a neighbourhood $U_{x_0}$ of $x_0$ such that for each $\tilde{x}\in U_{x_0}$ and $\tilde{y}=P^{-1}\tilde{x}$,
\begin{eqnarray}
\int_0^Te^{i\frac{2\pi n_{0}-\theta_{j_0}}{T}s}\tilde{y}_{j_0}(s)\mathrm{d}s\neq0.
\end{eqnarray}
Since $\Gamma_1$ is compact, there exist finite $U_{x_l}$ for $1\leq l\leq r$ such that
$$\Gamma_1\subset\cup_{l=1}^r U_{x_l}.$$
Let
$$\Upsilon(x)=(\Upsilon_1(x),\cdots,\Upsilon_r(x))^\top,$$
where $\Upsilon_l(x)=\int_0^Te^{i\frac{2\pi n_{l}-\theta_{j_l}}{T}s}y_{j_l}(s)\mathrm{d}s$, for $1\leq l\leq r$.
Then for each $x\in\Gamma_1$, $\Upsilon(x)\neq0$ and
$$\Upsilon_l(\mathcal{Q}(s)x)=e^{i\frac{-2\pi n_l+\theta_{j_l}}{T}s}\Upsilon_l(x),$$
for $1\leq l\leq r$.

Assume $\ind\Gamma_1=k$.
Then there exists a mapping
$$\Phi=(\Phi_1^\top,\cdots,\Phi_m^\top)^\top\in C(\Gamma_1,\mathbf{C}^{k}\setminus\{0\})$$
satisfying Definition \ref{def2}. Since $\Gamma_1$ is a closed set, by Tietze's theorem there exists a continuous extension $\tilde{\Phi}_j$ of $\Phi_j$ over $X$ for each $1\leq j\leq m$. Now define a mapping $\Psi=(\Psi_1^\top,\cdots,\Psi_m^\top)^\top$ by
\begin{eqnarray}\label{3.24}
\Psi_j(x)=\lim\limits_{t\rightarrow\infty}\frac{1}{t}\int_{0}^{t}e^{i\frac{-2\pi n_j-\theta_{p_j}}{T}s}\tilde{\Phi}_j(\mathcal{Q}(s)x)\mathrm{d}s.
\end{eqnarray}
Since $\tilde{\Phi}_j(\mathcal{Q}(s)x)$ is an almost periodic function on $s$, \eqref{3.24} is well defined.
Clearly, $\Psi$ is continuous, $\Psi(x)=\Phi(x)$ for $x\in\Gamma_1$ and
$$\Psi_j(\mathcal{Q}(s)x)=e^{i\frac{2\pi n_j+\theta_{p_j}}{T}s}\Psi(x)\ \ \forall x\in X,\ s\in\mathbf{R}.$$
Let
$$\Gamma_\delta=\{x\in X, \dist(x,\Gamma_1)\leq\delta\}.$$
Then it is easy to see that $\Gamma_\delta$ is $\mathcal{Q}(s)$ invariant. Since $\Psi$ is continuous and $\Psi(x)\neq0$ for $x\in\Gamma_1$, there exists a $\delta>0$ such that $\Psi(x)\neq0$ for $x\in\Gamma_\delta$ which yields $\ind \Gamma_\delta\leq k$. By monotonicity one has $\ind \Gamma_\delta\geq k$. Thus $\ind \Gamma_1=\ind \Gamma_\delta.$

{\rm (v)}
If $\ind\Gamma_1=\infty$ or $\ind\Gamma_2=\infty$, the result is obvious. Assume $\ind \Gamma_1=k_1<\infty$ and $\ind\Gamma_2=k_2<\infty$, then there exist
$$\Phi^j=((\Phi^j_1)^\top,\cdots,(\Phi^j_{m_j})^\top)^\top\in C(\Gamma_j,\mathbf{C}^{k_j}\setminus\{0\})$$
for $j=1,\ 2$ satisfying \eqref{2.4}. Similar to the proof of (iv), there exist continuous extensions $\Psi_j(x)$ of $\Phi^j(x)$ for $j=1, 2$ over $X$ satisfying \eqref{2.4}.  Define
$$\Psi:X\rightarrow\mathbf{C}^{k_1+k_2}$$
by
$$\Psi(x)=(\Psi_1^\top(x),\Psi_2^\top(x))^\top.$$
Then $\Psi(x)\neq0$ for every $x\in \Gamma_1\cup\Gamma_2$ and satisfies \eqref{2.4}, yielding
$\ind(\Gamma_1\cup\Gamma_2)\leq k_1+k_2$.
\end{proof}

A $\mathcal{Q}(s)$-orbit is the set $\{\mathcal{Q}(s)x: s\in\mathbf{R}\}$ for some $x\in X$. We have the following.
\begin{lemma}\label{p2}
Assume $\Gamma$ is an invariant subset of $X$ such that
$$\Gamma\cap\fix\{\mathcal{Q}(s)\}=\emptyset.$$
If $\ind \Gamma=k>0$, then there
exist at least $k$ $\mathcal{Q}(s)$-orbits on $\Gamma$.
\end{lemma}
\begin{proof}
Assume $\Gamma$ only contains $k-1$ orbits: $\mathcal{Q}(s)x^1, \cdots, \mathcal{Q}(s)x^{k-1}$.
As in the proof of Lemma \ref{p1}, for each $x^l\in\Gamma$ and $y^l=P^{-1}x^l$, there exist a component $y^l_{j_l}$ and $n_l\in\mathbf{Z}$ with
$2\pi n_l+\theta_{j_l}\neq0$ such that
$$\int_0^Te^{i\frac{-2\pi n_{l}-\theta_{j_l}}{T}s}y_{j_l}(s)\mathrm{d}s\neq0,$$
for $1\leq l\leq k-1$.
Now let
$$\Phi=(\Phi_1,\cdots,\Phi_{k-1})^\top,$$
where
$\Phi_l(x)=\int_0^Te^{i\frac{-2\pi n_{l}-\theta_{j_l}}{T}s}y_{j_l}(s)\mathrm{d}s$.
Then
$\Phi\in C(\Gamma,\mathbf{C}^{k-1}\setminus\{0\})$
and
$$\Phi_l(\mathcal{Q}(s)x)=e^{i\frac{2\pi n_{l}+\theta_{j_l}}{T}s}\Phi_l(x)\ \  \forall s\in\mathbf{R},\ x\in\Gamma.$$
Thus $\ind\Gamma\leq k-1$ and the assumption is not true, proving the lemma.
\end{proof}

Let $\mathfrak{B}$ denote a class of homeomorphisms $h: X\rightarrow X$ satisfying the following conditions,
\begin{enumerate}
\item [(a)] $h$ is equivariant and $h(0)=0$;
\item [(b)] given a compact set $\Gamma$ contained in a finite dimensional invariant space $Y$ and a constant $\varepsilon>0$, there exist a
finite dimensional invariant space $Z\supset Y$ and an  equivariant homeomorphism
$\tilde{h}: Z\rightarrow Z$ such that $\|h(z)-\tilde{h}(z)\|\leq\varepsilon$ for all $z\in\Gamma$.
\end{enumerate}
It is easy to see that if $h_1,\ h_2\in\mathfrak{B}$, then $h_1\circ h_2\in\mathfrak{B}$. Denote
$$\hat{X}=\left\{x\in X;\lim\limits_{t\rightarrow\infty}\frac{1}{t}\int_0^t x(s)\mathrm{d}s=0\right\},$$
Assume $E$ is a real continuous function on $X$. For $j\geq 1$, define
\begin{eqnarray*}
\mathcal{A}_j=\{A\subset X: \ A\ {\rm is}\ \mathcal{Q}(s)\ {\rm invariant\ and}\ \ind (h(A)\cap\hat{X})\geq j\ {\rm\ for\ all}\ h\in\mathfrak{B}\},
\end{eqnarray*}
\begin{eqnarray}\label{02.15}
c_j=\inf\limits_{A\in\mathcal{A}_j}\sup\limits_{A}E.
\end{eqnarray}
It is obviously that $\mathcal{A}_j\subset\mathcal{A}_{j-1}$ for $j\geq2$. Then
\begin{eqnarray}\label{2.14}
-\infty\leq c_1\leq c_2\leq\cdots\leq+\infty.
\end{eqnarray}
Set
$$K_c=\{x\in X: E'(x)=0\ \ {\rm and}\ E(x)=c\},$$
$$E^c=\{x\in X: E(x)\leq c\},$$
where $c$ is a constant.

Now we give an important property of the $\mathcal{Q}(s)$ index.
\begin{theorem}\label{th3.1}
Assume $E\in C^1(X,\mathbf{R})$ is $\mathcal{Q}(s)$ invariant and satisfies (P.-S.) (Palais-Smale condition). If $-\infty<c_j<+\infty$ and $K_{c_j}\cap\fix\{\mathcal{Q}(s)\}=\emptyset$, then $c_j$
is a critical value of $E$. Moreover, if $c_i=c_j$ for some $i\leq j$, then
$$\ind K_{c_i}\geq j-i+1.$$
\end{theorem}
Before proving Theorem \ref{th3.1}, we introduce the concept of ``pseudo-gradient".
\begin{definition}\label{def3.3}
Assume $Y$ is a Banach space, $\varphi\in C^1(Y,\mathbf{R})$ and
$$\tilde{Y}=\{u\in Y:\varphi'(u)\neq0\}.$$
A pseudo-gradient vector field for $\varphi$ on $\tilde{Y}$ is a locally Lipschitz continuous mapping $v:\tilde{Y}\rightarrow Y$ such that for every
$u\in\tilde{Y}$,
\begin{eqnarray*}
\|v(u)\|\leq2\|\varphi'(u)\|,
\end{eqnarray*}
\begin{eqnarray*}
\langle v(u),\varphi'(u)\rangle\geq\|\varphi'(u)\|^2.
\end{eqnarray*}
\end{definition}

We need the following.
\begin{lemma}\label{lem3.1}(see \cite{Mawhin}).
Under the assumption of Definition \ref{def3.3}, there exists a pseudo-gradient vector field for $\varphi$ on $\tilde{Y}$.
\end{lemma}
\begin{lemma}
Assume $E\in C^1(X,\mathbf{R})$ is $\mathcal{Q}(s)$ invariant, then there exists an equivariant pseudo-gradient vector filed for
$E$ on $\tilde{X}=\{z\in X: E'(z)\neq0\}$. That is, $v(\mathcal{Q}(s)z)=\mathcal{Q}(s)v(z)$ for every $z\in\tilde{X}$ and $s\in\mathbf{R}$.
\end{lemma}
\begin{proof}
By Lemme \ref{lem3.1}, there exists a pseudo-gradient vector field $w:\tilde{X}\rightarrow X$.
Define $v:\tilde{X}\rightarrow X$ by
$$v(z)=\lim\limits_{r\rightarrow\infty}\frac{1}{r}\int_0^r\mathcal{Q}(-s)w(\mathcal{Q}(s)z)\mathrm{d}s.$$
Obviously, $\mathcal{Q}(-s)w(\mathcal{Q}(s)z)$ is almost periodic in $s$, and $v$ is well defined.
Hence for $\tau\in\mathbf{R}$,
\begin{align*}
v(\mathcal{Q}(\tau)z)&=\lim\limits_{r\rightarrow\infty}\frac{1}{r}\int_{0}^r\mathcal{Q}(-s)w(\mathcal{Q}(s+\tau)z)\mathrm{d}s\\
&=\mathcal{Q}(\tau)\lim\limits_{r\rightarrow\infty}\frac{1}{r}\int_{0}^r\mathcal{Q}(-s-\tau)w(\mathcal{Q}(s+\tau)z)\mathrm{d}s\\
&=\mathcal{Q}(\tau)\lim\limits_{r\rightarrow\infty}\frac{1}{r}\int_{0}^r\mathcal{Q}(-t)w(\mathcal{Q}(t)z)\mathrm{d}t\\
&=\mathcal{Q}(\tau)v(z).
\end{align*}
By a simple calculation, we obtain
\begin{align*}
\|v(z)\|\leq\sup\limits_{s\in\mathbf{R}}\|w(\mathcal{Q}(s)z)\|\leq2\sup\limits_{s\in\mathbf{R}}\|E'(\mathcal{Q}(s)z)\|=2\|E'(z)\|,
\end{align*}
\begin{align*}
\langle v(z),E'(z)\rangle&=\lim\limits_{r\rightarrow\infty}\frac{1}{r}\int_{0}^r\left\langle \mathcal{Q}(-s)w(\mathcal{Q}(s)z),E'(z)\right\rangle\mathrm{d}s\\
&=\lim\limits_{r\rightarrow\infty}\frac{1}{r}\int_{0}^r\left\langle w(\mathcal{Q}(s)z),E'( \mathcal{Q}(s)z)\right\rangle\mathrm{d}s\\
&\geq\lim\limits_{r\rightarrow\infty}\frac{1}{r}\int_{0}^r\|E'(\mathcal{Q}(s)z)\|^2\mathrm{d}s\\
&=\|E'(z)\|^2.
\end{align*}
It suffices to show that $v$ is locally Lipschitz continuous.
For $z\in X$, denote $\Gamma=\{\mathcal{Q}(s)z:s\in\mathbf{R}\}$, then the closure $\overline{\Gamma}$ is the hull of $z$ and so is compact. Hence there exists a $\delta>0$ such that $w$ is Lipschitz continuous on
$$\Gamma_\delta=\{z\in X:\dist(z,\overline{\Gamma})<\delta\}.$$
Clearly, $\Gamma_\delta$ is $\mathcal{Q}(s)$ invariant, and for each $z_1,z_2\in\Gamma_\delta$ we have
\begin{align*}
\|v(z_1)-v(z_2)\|&\leq\lim\limits_{r\rightarrow\infty}\frac{1}{r}\int_{0}^r\|\mathcal{Q}(-s)(w(\mathcal{Q}(s)z_1)-w(\mathcal{Q}(s)z_2))\|\mathrm{d}s\\
&=\lim\limits_{r\rightarrow\infty}\frac{1}{r}\int_{0}^r\|(w(\mathcal{Q}(s)z_1)-w(\mathcal{Q}(s)z_2))\|\mathrm{d}s\\
&\leq L\lim\limits_{r\rightarrow\infty}\frac{1}{r}\int_{0}^r\|\mathcal{Q}(s)(z_1-z_2)\|\mathrm{d}s\\
&=L\|z_1-z_2\|,
\end{align*}
where $L$ is the Lipschitz constant of $w$ on $\Gamma_\delta$.
\end{proof}

\begin{lemma}\label{lem3}
Assume $E\in C^1(X,\mathbf{R})$ is $\mathcal{Q}(s)$ invariant and satisfies (P.-S.). Let $U$ be an open invariant neighbourhood of $K_c$. Then for each $\epsilon>0$, there exist $\varepsilon\in(0,\epsilon]$ and $\eta\in C([0,1]\times X,X)$
such that

\begin{enumerate}
\item [{\rm (i)}] $\eta(1,E^{c+\varepsilon}\setminus U)\subset E^{c-\varepsilon}$;
\item [{\rm (ii)}] if $z\notin E^{-1}([c-\epsilon,c+\epsilon])$, $\eta(t,z)=z$ for each $t\in[0,1]$;
\item [{\rm (iii)}] $\eta(t,\mathcal{Q}(s)z)=\mathcal{Q}(s)\eta(t,z),\ {\rm for\ each\ }z\in X,\ t\in[0,1],\ s\in\mathbf{R}$;
\item [{\rm (iv)}] if $E(0)=0$ and $c\neq 0$, $\eta(t,\cdot)\in\mathfrak{B}$ for each $t\in[0,1]$.
\end{enumerate}
\end{lemma}
\begin{proof}
First we claim that for each given $\epsilon>0$, there exists a $\varepsilon\in(0,\frac{\epsilon}{2}]$ such
that if $z\in E^{-1}([c-2\varepsilon,c+2\varepsilon])\cap (U^C)_{2\sqrt{\varepsilon}}$, then
\begin{align}\label{3.4}
\|E'(z)\|\geq4\sqrt{\varepsilon},
\end{align}
where $U^C$ denotes the complement of $U$, and $(U^C)_{2\sqrt{\varepsilon}}$ is the $2\sqrt{\varepsilon}$ neighbourhood of
$U^C$. If such a $\varepsilon$ does not exist, there is a sequence $\{z_k\}$ such that
$$z_k\in(U^C)_{\frac{2}{\sqrt{k}}},\ c-\frac{2}{k}\leq E(z_k)\leq c+\frac{2}{k},$$
and
$$\|E'(z_k)\|<\frac{4}{\sqrt{k}}.$$
Since $E$ satisfies (P.-S.), it has a convergent subsequence, without loss of generality, still denoted by $\{z_k\}$.
Assume $\lim\limits_{k\rightarrow\infty}z_k=z$, then
$$E(z)=c,\ \ E'(z)=0,$$
and $z\in K_c\cap U^C=\emptyset$, a contradiction.

Let
$$A=\{z\in X:z\in E^{-1}([c-2\varepsilon,c+2\varepsilon])\cap (U^C)_{2\sqrt{\varepsilon}}\},$$
$$B=\{z\in X:z\in E^{-1}([c-\varepsilon,c+\varepsilon])\cap (U^C)_{\sqrt{\varepsilon}}\}\subset A,$$
$$\psi(z)=\frac{\dist(z,A^C)}{\dist(z,A^C)+\dist(z,B)}.$$
Then
$0\leq \psi(z)\leq1,$
and $\psi(z)=1$ if $z\in B$, $\psi(z)=0$ if $z\in A^C$.
Define a continuous function $g$ on $X$ by
\begin{equation}
g(z)=\bigg\{
\begin{array}{ccc}
-\psi(z)\frac{v(z)}{\|v(z)\|},\ {\rm if}\ z\in A,\\
0,\ \ \ \ \ \ \ \ \ \ \ \ \ \ \  {\rm if}\ z\notin A,
\end{array}
\end{equation}
where $v(z)$ is an equivariant pseudo-gradient vector filed for $E$. Since $g$ is locally Lipschitz continuous and bounded, the following Cauchy problem has a unique solution $x(\cdot,z)$ defined on $[0,\infty)$:
\begin{equation}\label{3.6}
\bigg\{
\begin{array}{ccc}
x'=g(x),\\
x(0)=z.
\end{array}
\end{equation}
Let
\begin{align*}
\eta(t,z)=x(\sqrt\varepsilon t,z),\ \ t\in[0,1].
\end{align*}
Then $\eta(\cdot,\cdot)$ is continuous and $\eta(t,\cdot):X\rightarrow X$ is a homeomorphism for each $t\in[0,1]$.

Now we prove that $\eta$ satisfies (i)-(iv).

(i) Since
\begin{align*}
\|x(t,z)-z\|=\|\int_0^tg(x(s,z))\mathrm{d}s\|\leq\int_0^t\|g(x(s,z))\|\mathrm{d}s\leq t,
\end{align*}
we have
$$x(t,U^C)\subset(U^C)_{\sqrt\varepsilon}\ \ \ \forall t\in[0,\sqrt\varepsilon].$$
By the definition of $g$, for $t\in[0,\sqrt\varepsilon]$ and $x(t,z)\in A$, we obtain
\begin{align*}
\frac{\mathrm{d}}{\mathrm{d}t}E(x(t,z))&=\langle E'(x(t,z)),g(x(t,z))\rangle\\
&=\left\langle E'(x(t,z)),-\psi(x(t,z))\frac{v(x(t,z))}{\|v(x(t,z))\|}\right\rangle\\
&=\frac{-\psi(x(t,z))}{\|v(x(t,z))\|}\left\langle E'(x(t,z)),v(x(t,z))\right\rangle\\
&\leq\frac{-\psi(x(t,z))}{\|v(x(t,z))\|}\|E'(x(t,z)))\|^2\\
&\leq0.
\end{align*}
For $z\in E^{c+\varepsilon}\setminus U$, if $E(x(\tau,z))<c-\varepsilon$ for some $\tau\in[0,\sqrt\varepsilon]$,
then $E(x(\sqrt\varepsilon,z))<c-\varepsilon$, and $\eta(1,z)\in E^{c-\varepsilon}$. If such $\tau\in[0,\sqrt\varepsilon]$ doesn't exist, then
$$x(t,z)\in B\ \ \ \forall t\in[0,\sqrt\varepsilon].$$
From \eqref{3.4}, we have
\begin{align*}
E(x(\sqrt{\varepsilon},z))&=E(z)+\int_0^{\sqrt\varepsilon}\frac{\mathrm{d}}{\mathrm{d}t}E(x(t,z))\mathrm{d}t\\
&=E(z)+\int_0^{\sqrt\varepsilon}\langle E'(x(t,z)),g(x(t,z))\rangle\mathrm{d}t\\
&=E(z)+\int_0^{\sqrt\varepsilon}\left\langle E'(x(t,z)),-\frac{v(x(t,z))}{\|v(x(t,z))\|}\right\rangle\mathrm{d}t\\
&\leq c+\varepsilon-\int_0^{\sqrt\varepsilon}\frac{\|E'(x(t,z))\|^2}{\|v(x(t,z))\|}\mathrm{d}t\\
&\leq c-\varepsilon.
\end{align*}
Then $\eta(1,z)\in E^{c-\varepsilon}$, and $\eta(1,E^{c+\varepsilon}\setminus U)\subset E^{c-\varepsilon}$.

(ii) If $z\notin E^{-1}([c-\epsilon,c+\epsilon])$, we have $g(z)=0$, and hence
$\eta(t,z)=z$ for each $t\in[0,1]$.

(iii) It is easy to see that $\psi(\mathcal{Q}(s)z)=\psi(z)$, which yields that $g(\mathcal{Q}(s)z)=\mathcal{Q}(s)g(z)$.
Then
\begin{align*}
\eta(t,\mathcal{Q}(s)z)=x(\sqrt\varepsilon t,\mathcal{Q}(s)z)&=\mathcal{Q}(s)z+\int_0^{\sqrt\varepsilon t}g(x(\tau,\mathcal{Q}(s)z))\mathrm{d}\tau\\
&=\mathcal{Q}(s)z+\mathcal{Q}(s)\int_0^{\sqrt\varepsilon t}g(\mathcal{Q}(-s)x(\tau,\mathcal{Q}(s)z))\mathrm{d}\tau.\\
\end{align*}
Thus
\begin{align*}
\mathcal{Q}(-s)\eta(t,\mathcal{Q}(s)z)=\mathcal{Q}(-s)x(\sqrt\varepsilon t,\mathcal{Q}(s)z)=z+\int_0^{\sqrt\varepsilon t}g(\mathcal{Q}(-s)x(\tau,\mathcal{Q}(s)z))\mathrm{d}\tau.
\end{align*}
Since the solution of \eqref{3.6} is unique, we get
$$\mathcal{Q}(-s)x(t,\mathcal{Q}(s)z)=x(t,z),$$
and thus
$$\eta(t,\mathcal{Q}(s)z)=\mathcal{Q}(s)\eta(t,z),\ \ {\rm for\ each}\ t\in[0,1],\ s\in\mathbf{R}.$$

(iv) Taking $\epsilon<|c|$, we have $0\notin E^{-1}(c-\epsilon,c+\epsilon)$. It follows from (ii) that
$\eta(t,0)=0$ for $t\in[0,1]$. Now we only need to prove that $\eta(t,\cdot)$ satisfies (b) of $\mathfrak{B}$.
Let $X_k\subset X$, $k\in\mathbf{N}$ be a sequence of finite dimensional $\mathcal{Q}(s)$ invariant subspaces such that
$X_0=Y$, $X_k\subset X_{k+1}$ and $\cup_{k=0}^\infty X_k=X$. Let $P_k$ be the orthogonal projection on $X_k$ and
$g_k=P_k\circ g\circ P_k$. Then it is easy to see that $g_k$ is locally Lipschitz continuous and $g_k(x)\rightarrow g(x)$ for $k\rightarrow\infty$ and each $x\in X$.
Let $x^k(t,z)$ be the solution of the following equation
\begin{equation}
\bigg\{
\begin{array}{ccc}
(x^k)'=g_k(x^k),\\
x^k(0)=z.
\end{array}
\end{equation}
It is easy to see that $x(t,\cdot)\rightarrow x^k(t,\cdot)$ uniformly on compact sets for a fixed $t\in[0,\sqrt{\varepsilon}]$.
For a given compact set $\Gamma$ and constant $\tilde{\varepsilon}>0$, take $k$ big enough so that
$$\|x(\sqrt{\varepsilon}t,z)-x^k(\sqrt{\varepsilon}t,z)\|\leq\tilde{\varepsilon}\ {\rm \ for}\ t\in[0,1],\ z\in\Gamma.$$
Then the proof is completed by choosing $Z=X_k$ and $\tilde{\eta}(t,\cdot)=x^k(\sqrt{\varepsilon}t,\cdot)$.
\end{proof}

We also need the following.
\begin{lemma}\label{lem2.3}
Assume $A\in\mathcal{A}_k$, $B$ is $\mathcal{Q}(s)$ invariant and $\ind B=r\leq k$. Then $A\setminus B\in \mathcal{A}_{k-r}$.
\end{lemma}
\begin{proof}
Since $(A\setminus B)\cap h^{-1}(\hat{X})=(A\cap h^{-1}(\hat{X}))\setminus B$ for every $h\in\mathfrak{B}$, by Lemma \ref{p1} we have
\begin{align*}
\ind(h(A\setminus B)\cap\hat{X})&=\ind((A\setminus B)\cap h^{-1}(\hat{X}))\\
&=\ind((A\cap h^{-1}(\hat{X}))\setminus B)\\
&\geq \ind(A\cap h^{-1}(\hat{X}))-\ind B\\
&=\ind(h(A)\cap\hat{X})-\ind B\\
&=k-r.
\end{align*}
\end{proof}

Now let us give the proof of Theorem \ref{th3.1}.
\begin{proof}[Proof of Theorem \ref{th3.1}]

Assume $c_i=c_j=c$ for some $1\leq i\leq j$. By Lemma \ref{p2}, it suffices to prove
$$\ind K_c\geq j-i+1.$$
Assume $\ind K_c=r<j-i+1$.
Since $E$ satisfies (P.-S.) and is $\mathcal{Q}(s)$ invariant, $K_c$ is compact and $\mathcal{Q}(s)$ invariant. By the continuity of index, there exists a $\mathcal{Q}(s)$ invariant neighbourhood $U$ of $K_c$
such that
$$\ind U=\ind K_c.$$
Take $A\in\mathcal{A}_j$ such that
$$\sup\limits_{A}E\leq c+\varepsilon,$$
and denote $\Omega=A\setminus U$.
Then by Lemma \ref{lem2.3}, one has $\Omega\in\mathcal{A}_{j-r}$.
Now we apply Lemma \ref{lem3} with $\Lambda=\eta(1,\Omega)\subset E^{c-\varepsilon}$. Then $\Lambda$ is $\mathcal{Q}(s)$ invariant, and
\begin{align}\label{2.16}
\sup\limits_{\Lambda}E\leq c-\varepsilon.
\end{align}
Since $\eta(1,\cdot)\in\mathfrak{B}$, one has $\Lambda\in\mathcal{A}_{j-r}$. Thus
by the definition of $c$, one deduces
$$\sup\limits_{\Lambda}E(x)\geq c_{j-r}=c,$$
which contradicts \eqref{2.16}, proving the theorem.
\end{proof}

We need the following.
\begin{lemma}\label{lem2.6}
Assume $V$ is a $\nu$ dimensional invariant subspace of $X$ such that $V\cap\fix\{\mathcal{Q}(s)\}=\{0\}$, and $D\subset V$ is a bounded neighborhood of $0$ in $V$. Then
$$\ind( \partial D)=\frac{\nu}{2},$$
where $\partial D$ is the boundary of $D$ relative to $V$.
\end{lemma}
To prove the lemma, we need the following lemma about $S^1$ acting on $\mathbf{R}^{2k}$.
\begin{lemma}[See Theorem 5.5 of \cite{Mawhin}]\label{lem2.7}
Let $\{T(\theta)\}_{\theta\in S^1}$ be an action of $S^1$ over $\mathbf{R}^{2k}$ such that $\fix(S^1)=\{0\}$ and let $D$ be an open
bounded invariant neighbourhood of $0$. If $\Phi\in C(\partial D,\mathbf{C}^{k-1})$  and $n\in \mathbf{N}\setminus\{0\}$ with
$$\Phi(T(\theta)z)=e^{in\theta}\Phi(z),\ \ \theta\in S^1,\ \ z\in\partial D,$$
then $0\in\Phi(\partial D)$.
\end{lemma}

Now we are going to prove Lemma \ref{lem2.6}.
\begin{proof}[Proof of Lemma \ref{lem2.6}]
Since $V\subset X$ is $\mathcal{Q}(s)$ invariant and $V\cap\fix\{\mathcal{Q}(s)\}=\{0\}$, it is easy to see that $\nu$ is even.
Let  $(z_1^\top,\cdots,z_{\nu/2}^\top)^\top$ be a set of  coordinates in $V$ such that
$$\mathcal{Q}(s)(z_1^\top,\cdots,z_{\nu/2}^\top)^\top=((Q_1(s)z_1)^\top,\cdots,(Q_{\nu/2}(s)z_{\nu/2})^\top)^\top,$$
where $z_j\in\mathbf{R}^2$ and
\begin{equation}
\begin{split}
Q_j(s)&=\left(
\begin{array}{ccc}
\cos \frac{2\pi m_j+\theta_{p_j}}{T}s&\sin \frac{2\pi m_j+\theta_{p_j}}{T}s\\
-\sin \frac{2\pi m_j+\theta_{p_j}}{T}s&\cos \frac{2\pi m_j+\theta_{p_j}}{T}s
\end{array}
\right),
\end{split}
\end{equation}
for $1\leq j\leq \frac{\nu}{2}$, $1\leq p_j\leq n$ and $m_j\in\mathbf{Z}$ with $\theta_{p_j}+2\pi m_j\neq0$.
We claim that for each set $A\subset V\setminus\{0\}$, one has $\ind A\leq \frac{\nu}{2}$. Clearly, there exists unitary matrix $P_0$
such that
$$P_0^{-1}Q_j(s)P_0=\diag\left\{e^{i\frac{2\pi m_j+\theta_{p_j}}{T}s},e^{-i\frac{2\pi m_j+\theta_{p_j}}{T}s}\right\},\ \ 1\leq j\leq \frac{\nu}{2}.$$
For each $z\in A$, denote $y=(y_{1,1},y_{1,2},\cdots, y_{\nu/2,1},y_{\nu/2,2})^\top=\bar{P}^{-1}z$,
where $\bar{P}=\diag\{P_0,\cdots,P_0\}$. Let
$$\Psi(z)=(y_{1,1},y_{2,1}\cdots,y_{\nu/2,1})^\top.$$
Then $\Psi\in C(A, \mathbf{C}^{\nu/2}\setminus\{0\})$ satisfies \eqref{2.4}. Assume $\ind(\partial D)=q<\frac{\nu}{2}$. Then there exists a mapping
\begin{align}\label{2.17}
\Phi=(\Phi_{1}^\top,\cdots,\Phi_{m}^\top)^\top\in C( \partial D,\mathbf{C}^q\setminus\{0\})
\end{align}
such that $\Phi_j=(\Phi_{j,1},\cdots,\Phi_{j,q_j})^\top\in C( \partial D,\mathbf{C}^{q_j})$ with $q_1+\cdots q_m=q$ and
$$\Phi_j(\mathcal{Q}(s)x)=e^{i\frac{2k_j\pi+\theta_{r_j}}{T}s}\Phi_j(x)$$
for $1\leq j\leq m,\ k_j\in\mathbf{Z},\ 1\leq r_j\leq n$. Rewrite $(z_1,\cdots,z_\frac{\nu}{2})=(\tilde{z}_1,\cdots,\tilde{z}_l)$
such that
$$\mathcal{Q}(s+T_j)\tilde{z}_j=\tilde{z}_j\ \ \forall s\in\mathbf{R},\ \ 1\leq j\leq l,$$
where $T_j>0$ is a constant and
\begin{align}\label{3.16}
k_1T_j+k_2T_r\neq0\ \  \forall j\neq r, \ \ \ k_1,\ k_2\in\mathbf{Z}\setminus\{0\}.
\end{align}
Then for any $x^j=(\tilde{x}^j_1,\cdots,\tilde{x}^j_l)\in\partial D$ with $\tilde{x}^j_a=0$, $1\leq a\leq l$, $a\neq j$, one has
$$\mathcal{Q}(s+T_j)x^j=x^j,\ \ \ 1\leq j\leq l.$$
Since
$$\Phi_a(\mathcal{Q}(T_j)x^j)=\Phi_a(x^j)=e^{i\frac{2k_a\pi+\theta_{r_a}}{T}T_j}\Phi_a(x^j)\ \ {\rm for}\ \ 1\leq a\leq m,$$
one deduces
$$\Phi_a(x^j)=0,\ \ \  {\rm if}\ \frac{(2k_a\pi+\theta_{r_a})T_j}{T}\neq2\pi k\ \ \ \forall k\in\mathbf{Z}.$$
Since $\Phi_a$ not always vanishes,  there exists a $1\leq j_a\leq l$ such that
$$\frac{(2k_a\pi+\theta_{r_a})T_{j_a}}{T}=2\pi n_a$$
for some $n_a\in\mathbf{Z}$. Let
$$\Upsilon=(\Phi^{h_1}_1,\cdots,\Phi^{h_m}_m)^\top=(\Upsilon_1^\top,\cdots,\Upsilon_l^\top)^\top$$
such that
$\Upsilon_j(\mathcal{Q}(s)x)=e^{i\alpha_j s}\Upsilon_j(x),$
$\alpha_j T_j=2p_j\pi$ for some $p_j\in\mathbf{Z}\setminus\{0\}$, $1\leq j\leq l$,
where $\Phi_k^{h_k}=(\Phi_{k,1}^{h_k},\cdots,\Phi_{k,q_k}^{h_k})^\top$, $h_k\in\mathbf{N}_+$ for $1\leq k\leq m$.
Then $\Upsilon_j(x^r)=0$ if $j\neq r$.
It follows from $q<\frac{\nu}{2}$ and \eqref{3.16} that there must exist a $1\leq j\leq m$ such that the dimension of $\Upsilon_j$
is less than the dimension of $x^{j}$. By Lemma \ref{lem2.7}, there exists a $x^{j}_0\in\partial D$ such that $\Upsilon_j(x^j_0)=0$.
Then $\Upsilon(x^j_0)=0$ and  $\Phi(x^{j}_0)=0$
which contradicts \eqref{2.17}, proving the lemma.
\end{proof}

Let $X^+$ denote the vector space spanned by the eigenfunctions corresponding to eigenvalues of $L$ in the set
$$\left\{\lambda_m^j=\mu_j-\bigg(\frac{\theta_j+2\pi m}{T}\bigg)^2>0:\ \theta_j+2\pi m\neq0,\ 1\leq j\leq n,\ m\in\mathbf{Z}\right\}.$$
Denote
$$S_\rho=\{x\in X; \|x\|=\rho\},\ {\rm for\ each}\ \rho>0.$$
\begin{lemma}\label{lem2.9}
$\ind (h(X^+\cap S_\rho)\cap \hat{X})\geq\frac{p_T}{2}$ for every $h\in\mathfrak{B}$ and $\rho>0$.
\end{lemma}
\begin{proof}
By property (iv) of the index, there exists a $\delta>0$ small enough so that
$$\ind(h(X^+\cap S_\rho)\cap\hat{X})=\ind(N_\delta(h(X^+\cap S_\rho)\cap\hat{X})),$$
where $N_\delta(\Gamma)$ denotes the $\delta$ neighbourhood of $\Gamma$ for $\Gamma\subset X$.
We claim that there exists a $\varepsilon>0$ small enough so that
\begin{align}\label{2.21}
N_\varepsilon(h(X^+\cap S_\rho))\cap\hat{X}\subset N_\delta(h(X^+\cap S_\rho)\cap\hat{X}).
\end{align}
If not, for each $k\in\mathbf{N}_+$ there exists a $z_k\in\hat{X}$ such that
$$\dist(z_k, h(X^+\cap S_\rho))\leq \frac{1}{k},$$
$$\dist(z_k, h(X^+\cap S_\rho)\cap\hat{X})>\delta.$$
Since $h(X^+\cap S_\rho)$ is compact, there exists a subsequence $\{z_{k_j}\}$
of $\{z_k\}$ such that $z_{k_j}\rightarrow \bar{z}$ for $j\rightarrow\infty$.
Taking the limit yields
$$\bar{z}\in h(X^+\cap S_\rho)\cap\hat{X}\ \ \ \ {\rm and}\ \ \ \ \dist(\bar{z}, h(X^+\cap S_\rho)\cap\hat{X})>\delta.$$
This is a contradiction. Let $\hat{X}^\bot$ be the orthogonal complement space of $\hat{X}$. It is easy to see $\hat{X}^\bot=\fix\{\mathcal{Q}(s)\}$.
Since $h\in\mathfrak{B}$,  there exist a finite dimensional invariant space $Z\supset X^+\oplus\hat{X}^\bot$ and
an equivariant  homeomorphism
$\tilde{h}: Z\rightarrow Z$ such that
$$\|h(z)-\tilde{h}(z)\|\leq\varepsilon$$
for all $z\in X^+\cap S_\rho$. It follows that
$$\tilde{h}(X^+\cap S_\rho)\cap\hat{X}\subset N_\varepsilon(h(X^+\cap S_\rho))\cap\hat{X}.$$
Hence
\begin{align*}
\ind (\tilde{h}(X^+\cap S_\rho)\cap\hat{X})&\leq\ind( N_\varepsilon(h(X^+\cap S_\rho))\cap\hat{X})=\ind(h(X^+\cap S_\rho)\cap\hat{X}).
\end{align*}
Now we only need to prove $\ind (\tilde{h}(X^+\cap S_\rho)\cap\hat{X})\geq\frac{p_T}{2}$.

Take $\delta>0$ small enough so that
$$\ind( N_\delta(\tilde{h}(X^+\cap S_\rho)\cap\hat{X}))=\ind(\tilde{h}(X^+\cap S_\rho)\cap\hat{X}).$$
Since
$$N_\delta^Z(\tilde{h}(X^+\cap S_\rho)\cap\hat{X})\subset N_\delta(\tilde{h}(X^+\cap S_\rho)\cap\hat{X}),$$
we have
$$\ind( N_\delta^Z(\tilde{h}(X^+\cap S_\rho)\cap\hat{X})=\ind(\tilde{h}(X^+\cap S_\rho)\cap\hat{X}),$$
where  $N_\delta^Z(\Gamma)$ denotes the $\delta$ neighbourhood of $\Gamma$ in $Z$.
As the proof of \eqref{2.21}, there exists a $\tilde{\varepsilon}>0$ small enough so that
$$\tilde{h}(N_{\tilde{\varepsilon}}^Z (X^+)\cap S_\rho)\cap\hat{X}\subset N_\delta^Z(\tilde{h}(X^+\cap S_\rho)\cap\hat{X}).$$
Thus
$$\ind(\tilde{h}(N_{\tilde{\varepsilon}}^Z (X^+)\cap S_\rho)\cap\hat{X})=\ind(\tilde{h}(X^+\cap S_\rho)\cap\hat{X}).$$
Let $Z=X^+\oplus \hat{X}^\bot\oplus Y$, $P^\bot$ denote the orthogonal projection on $Y$ and
$$\tilde{\Gamma}=Z\setminus N_{\tilde{\varepsilon}}^Z(X^+).$$
Clearly, $P^\bot$ is equivariant and $\fix\{\mathcal{Q}(s)\}\cap Y=\{0\}$. It follows that
$$\ind \tilde{\Gamma}\leq\ind(P^\bot \tilde{\Gamma})\leq \frac{1}{2}\dim Y=\frac{1}{2}(\dim Z-\dim X^+-\dim\hat{X}^\bot).$$
Since
$$\tilde{h}(Z\cap S_\rho)=\tilde{h}(N_{\tilde{\varepsilon}}^Z(X^+)\cap S_\rho)\cup\tilde{h}(\tilde{\Gamma}\cap S_\rho),$$
one has
\begin{align}\label{2.22}
\ind(\tilde{h}(Z\cap S_\rho)\cap\hat{X})&\leq\ind(\tilde{h}(N_{\tilde{\varepsilon}}^Z(X^+)\cap S_\rho)\cap\hat{X})+\ind(\tilde{h}(\tilde{\Gamma}\cap S_\rho)\cap\hat{X})\nonumber\\
&=\ind(\tilde{h}(X^+\cap S_\rho)\cap\hat{X})+\ind(\tilde{h}(\tilde{\Gamma}\cap S_\rho)\cap\hat{X}).
\end{align}
By the mapping and monotonicity properties of the index, we have
\begin{align}\label{2.23}
\ind(\tilde{h}(\tilde{\Gamma}\cap S_\rho)\cap\hat{X})&\leq \ind(\tilde{h}(\tilde{\Gamma}\cap S_\rho))=\ind(\tilde{\Gamma}\cap S_\rho)\nonumber\\
&\leq\ind \tilde{\Gamma}\leq\frac{1}{2}(\dim Z-\dim X^+-\dim\hat{X}^\bot).
\end{align}
Since $\tilde{h}(0)=0$ and $\tilde{h}$ is a homeomorphism, $\tilde{h}(Z\cap B_\rho )$ is  a neighbourhood of $0$ in $Z$
and $\partial(\tilde{h}(Z\cap B_\rho))=\tilde{h}(Z\cap S_\rho)$,
where $B_\rho$ denotes the ball with radius $\rho$.
By Lemma \ref{lem2.6}, we obtain
\begin{align}\label{2.24}
\ind(\tilde{h}(Z\cap S_\rho)\cap \hat{X})&=\ind(\tilde{h}(Z\cap S_\rho)\cap( \hat{X}\cap Z))\nonumber\\
&=\frac{1}{2}\dim(\hat{X}\cap Z)\nonumber\\
&=\frac{1}{2}(\dim Z-\dim \hat{X}^\bot).
\end{align}
It follows from \eqref{2.22}-\eqref{2.24} that
\begin{align*}
\ind(\tilde{h}(X^+\cap S_\rho)\cap\hat{X})&\geq\frac{1}{2}(\dim Z-\dim \hat{X}^\bot)-\frac{1}{2}(\dim Z-\dim X^+-\dim \hat{X}^\bot)\\
&=\frac{1}{2}\dim X^+\\
&=\frac{p_T}{2},
\end{align*}
proving the lemma.
\end{proof}

\section{Proof of Theorem \ref{theorem1}}
First we show a property of rotating quasi-periodic functions.
\begin{lemma}\label{lemma3.1}
Assume $x\in X$ and $\lim\limits_{t\rightarrow\infty}\frac{1}{t}\int_0^tx(s)\mathrm{d}s=0$. Then
\begin{align}
\int_0^T|x'(t)|^2\mathrm{d}t\geq M_0\int_0^T|x(t)|^2\mathrm{d}t,
\end{align}
where $M_0=\min\limits_{1\leq j\leq n}\left\{\big|\frac{\theta_j+2\iota\pi}{T}\big|^2\neq0;\ \iota=-1,\ 0\right\}$.
\end{lemma}
\begin{proof}
Consider the function
$$z(t)=\diag\{e^{-\frac{i\theta_1t}{T}},\cdots,e^{-\frac{i\theta_nt}{T}}\}y(t),$$
where $y(t)$ is defined in \eqref{02.5}. Then  $z(t+T)=z(t)$ for all $t\in\mathbf{R}$, and
$z(t)=\sum\limits_{j\in\mathbf{Z}}\bar{z}_j e^{i \frac{2j\pi t}{T}}$ with $\bar{z}_j\in\mathbf{C}^n$. It follows that
\begin{align*}
x(t)&=P\left(e^{\frac{i\theta_1t}{T}}z_1(t),\cdots,e^{\frac{i\theta_nt}{T}}z_n(t)\right)^\top\\
&=\sum\limits_{j\in\mathbf{Z}}P\left(e^{\frac{i(2\pi j+\theta_1)t}{T}}\bar{z}_{1,j},\cdots,e^{\frac{i(2\pi j+\theta_n)t}{T}}\bar{z}_{n,j}\right)^\top,
\end{align*}
where $z_m(t)$ is the $m$ component of $z(t)$ and $\bar{z}_{m,j}$ the $m$ component of $\bar{z}_j$ for $1\leq m\leq n$ and $j\in\mathbf{Z}$.
Since
$$\lim\limits_{t\rightarrow\infty}\frac{1}{t}\int_0^tx(s)\mathrm{d}s=0,$$
one has $\bar{z}_{m,0}=0$, if $\theta_m=0$ for some $1\leq m\leq n$.
Then
\begin{align*}
\int_0^T|x'(t)|^2\mathrm{d}&t=T\sum\limits_{j\in\mathbf{Z}}\sum\limits_{1\leq m\leq n}\left(\frac{2\pi j+\theta_m}{T}\right)^2|\bar{z}_{m,j}|^2\\
&\geq M_0T\sum\limits_{j\in\mathbf{Z}}\sum\limits_{1\leq m\leq n}|\bar{z}_{m,j}|^2\\
&=M_0\int_0^T|x(t)|^2\mathrm{d}t.
\end{align*}

\end{proof}

Consider the following functional on $X$:
$$E(x)=\int_0^T\left(\frac{1}{2}|x'(t)|^2-V(x(t))\right)\mathrm{d}t.$$
It is easy to see that each critical point $x$ of $E$ on $X$ is a $(Q,T)$-rotating quasi-periodic solution of system \eqref{1.1}.
Clearly, $E$ is $\mathcal{Q}(s)$ invariant and  we will show that $E$ satisfies (P.-S.) condition.
\begin{lemma}
Assume (V1)-(V5) hold. Then $E$ satisfies (P.-S.) condition.
\end{lemma}
\begin{proof}
Assume $\{x_k;\ k\in\mathbf{N}_+\}$ is a consequence of $X$ such that
\begin{align}\label{03.18}
|E(x_k)|\leq M,\ \ k\in\mathbf{N_+}
\end{align}
for some constant $M>0$, and
\begin{align}\label{3.18}
\lim\limits_{k\rightarrow\infty}\|E'(x_k)\|=0.
\end{align}
We will show that $\{x_k\}$ has a converging subsequence.
For each $x\in X$, one has $x=\bar{x}+\tilde{x}$,
where
$$\bar{x}=\lim\limits_{t\rightarrow\infty}\frac{1}{t}\int_0^tx(s)\mathrm{d}s,\ \ \ \ \ \tilde{x}=x-\bar{x}.$$
By \eqref{3.18}, for $k$ large enough,
$$\langle E'(x_k),\tilde{x}_k\rangle\leq\|\tilde{x}_k\|,$$
which implies
$$\int_0^T|\tilde{x}'_k(t)|^2\mathrm{d}t\leq\|\tilde{x}_k\|+\int_0^T\left(\nabla V(x_k(t)),\tilde{x}_k(t)\right)\mathrm{d}t.$$
By (V4), there exist constants $M_1,\ M_2>0$ such that
\begin{align}\label{3.20}
\int_0^T|\tilde{x}'_k(t)|^2\mathrm{d}t&\leq\|\tilde{x}_k\|+M_1\int_0^T|\tilde{x}_k(t)|\mathrm{d}t\nonumber\\
&\leq M_2\|\tilde{x}_k\|.
\end{align}
It follows from Lemma \ref{lemma3.1} and \eqref{3.20} that there exist a constant $M_3>0$ such that
$\|\tilde{x}_k\|\leq M_3$ for $k$ large enough. Then there exists a constant $M_4>0$ such that
$$\int_0^T|\tilde{x}_k(t)|\mathrm{d}t\leq M_4.$$
By \eqref{03.18} and (V4), we have
\begin{align}
M&\geq\left|\int_0^T\left(\frac{1}{2}|\tilde{x}_k'|^2-V(\bar{x}_k+\tilde{x}_k)\right)\mathrm{d}t\right|\nonumber\\
&\geq\left|\int_0^TV(\bar{x}_k)\mathrm{d}t\right|-M_5\int_0^T|\tilde{x}_k|\mathrm{d}t-\frac{M_3^2}{2}\nonumber\\
&\geq TV(\bar{x}_k)-M_5M_4-\frac{M_3^2}{2}
\end{align}
for some constant $M_5>0$. Then $V(\bar{x}_k)$ is uniformly  bounded.
It is easy to see that $\bar{x}_k\in\ker(I-Q)$, and by (V5), $\bar{x}_k$ is uniformly bounded.
Hence $\|x_k\|$ is bounded and there exists a subsequence $\{x_{k_j}\}$ of $\{x_{k}\}$
such that $\bar{x}_{k_j}\rightarrow\bar{x}$ and $x_{k_j}\rightharpoonup x$ for $j\rightarrow\infty$ and some $x\in X$.
By \eqref{3.18}, for each $\varepsilon>0$ there exists a $j_0>0$, such that for $j\geq j_0$, one has
\begin{align}
\left|\int_0^T(x'_{k_j},x'_{k_j}-x')\mathrm{d}t\right|&\leq\left|\int_0^T(\nabla V(x_{k_j}),x_{k_j}-x)\mathrm{d}t\right|+
\varepsilon\|x_{k_j}-x\|.
\end{align}
It is well known that the operator
$x\in H^1[0,T]\rightarrow V'(x)\in H^{-1}[0,T]$ is compact and we have
$$\lim\limits_{j\rightarrow\infty}\int_0^T\left(\nabla V(x_{k_j}),x_{k_j}-x\right)\mathrm{d}t=0.$$
Then for $j$ large enough,
\begin{align*}
\left|\int_0^T(x'_{k_j},x'_{k_j}-x')\mathrm{d}t\right|&\leq\varepsilon+
\varepsilon\|x_{k_j}-x\|\leq(1+M_6)\varepsilon,
\end{align*}
for some constant $M_6>0$. Thus there exists a constant $M_7>0$ such that
\begin{align*}
\|\tilde{x}_{k_j}-\tilde{x}\|&\leq M_7\int_0^T|x'_{k_j}-x'|^2\mathrm{d}t\\
&\leq M_7\left|\int_0^T(x'_{k_j},x'_{k_j}-x')\mathrm{d}t\right|+M_7\int_0^T|(x',x'_{k_j}-x')|\mathrm{d}t\\
&\leq M_7(1+M_6)\varepsilon+M_7\int_0^T|(x',x'_{k_j}-x')|\mathrm{d}t.
\end{align*}
Since $x_{k_j}\rightharpoonup x$ for $j\rightarrow\infty$ on $X$, one has
$$\lim\limits_{j\rightarrow\infty}\|\tilde{x}_{k_j}-\tilde{x}\|=0,$$
proving the lemma.
\end{proof}

Now we give the proof of Theorem \ref{theorem1}.
\begin{proof}[Proof of Theorem \ref{theorem1}]
It follows from (V1) that
$$V(x)=\frac{1}{2}\left(x, V_{xx}(0)x\right)+o(|x|^2).$$
Then for $x\in X$,
\begin{align*}
E(x)&=\int_0^T\left(\frac{1}{2}|x'|^2-\frac{1}{2}\left(x, V_{xx}(0)x\right)+o(|x|^2)\right)\mathrm{d}t\\
&\leq-\frac{1}{2}\int_0^T(Lx,x)\mathrm{d}t+o(\|x\|^2).
\end{align*}
Denote
$$\nu=\min\{\lambda_k^j;\ \lambda_k^j>0,\ \theta_j+2\pi k\neq0,\ 1\leq j\leq n,\ k\in\mathbf{Z}\}.$$
For $x\in X^+$,
we have
$$E(x)\leq -\nu\|x\|^2+o(\|x\|^2).$$
Hence there exist constants $\gamma<0$ and $\rho>0$ such that
$$E(x)\leq\gamma\ \ \ {\rm for} \ \ x\in X^+\cap S_\rho.$$
By Lemma \ref{lem2.9}, $X^+\cap S_\rho\in\mathcal{A}_{p_T/2}$. Thus
\begin{eqnarray}\label{2.14}
-\infty\leq c_1\leq\cdots\leq c_{\frac{p_T}{2}}\leq \gamma.
\end{eqnarray}
Now if we can prove that $c_j>-\infty$ and $K_{c_j}\cap\fix\{\mathcal{Q}(s)\}=\emptyset$ for $1\leq j\leq \frac{p_T}{2}$,
the proof is completed by Theorem \ref{th3.1} and Lemma \ref{p2}.
In fact, if $x\in K_{c_j}\cap\fix\{\mathcal{Q}(s)\}$, then $x$ is a constant and $x\in\ker(I-Q)$.
By (V2), $V(x)<0$ and $E(x)>0$, which contradicts \eqref{2.14}.
It follows from (V4) that there exists a constant $M_8>0$ such that
$$V(x)\leq M_8+\frac{M_0}{2}|x|^2,\ \ x\in\mathbf{R}^{n}.$$
Then by Lemma \ref{lemma3.1}, for $x\in \hat{X}$ we have
\begin{align*}
E(x)&\geq\int_0^T \left(\frac{M_0}{2}|x(t)|^2-V(x(t))\right)\mathrm{d}t\\
&\geq\int_0^T \left(\frac{M_0}{2}|x(t)|^2-\big(\frac{M_0}{2}|x|^2+M_8\big)\right)\mathrm{d}t\\
&=-M_8T.
\end{align*}
Since $\ind(A\cap\hat{X})\geq j$ for each $A\in\mathcal{A}_j$, one has
$A\cap\hat{X}\neq\emptyset$. Therefore,
$$\sup_A E\geq\sup_{A\cap\hat{X}}E\geq\inf_{\hat{X}}E\geq-M_8T$$
for $1\leq j\leq\frac{p_T}{2}$. It follows that $c_j\geq-M_8T$ for
$1\leq j\leq\frac{p_T}{2}$,
proving the theorem.
\end{proof}

\section{Proof of Theorem \ref{theorem2}}
In this section we give the proof of Theorem \ref{theorem2}.
\begin{proof}[Proof of Theorem \ref{theorem2}]
First we prove that under assumptions (V1), (V2), (V6), (V7), the functional $E$ satisfies (P.-S.) condition.
It follows from (V6) that there exist constants $a_3,\ a_4>0$ such that
$$V(x)\leq a_3|x|^\beta+a_4.$$
Assume $\{x_k;\ k\in\mathbf{N}_+\}$ is a sequence of $X$ such that
\begin{align}\label{4.23}
|E(x_k)|\leq M,\ \ k\in\mathbf{N_+}
\end{align}
for some constant $M>0$, and
\begin{align}\label{4.24}
\lim\limits_{k\rightarrow\infty}\|E'(x_k)\|=0.
\end{align}
Then for $k$ large enough,
$$\langle E'(x_k),x_k\rangle\leq\|x_k\|,$$
and hence
$$\int_0^T|x'_k|^2\mathrm{d}t\leq\int_0^T\left(\nabla V(x_k),x_k\right)\mathrm{d}t+\|x_k\|.$$
By (V6), we have
$$\int_0^T|x'_k|^2\mathrm{d}t\leq\beta\int_0^TV(x_k)\mathrm{d}t+a_5+\|x_k\|,$$
for some constant $a_5>0$. By \eqref{4.23}, we deduce
\begin{align}\label{4.28}
\int_0^TV(x_k)\mathrm{d}t\leq\frac{1}{2}\int_0^T|x'_k|^2\mathrm{d}t+M.
\end{align}
Then
$$\int_0^T|x'_k|^2\mathrm{d}t\leq\frac{\beta}{2}\int_0^T|x'_k|^2\mathrm{d}t+\beta M+a_5+\|x_k\|,$$
which implies
\begin{align}\label{4.29}
(1-\frac{\beta}{2})\int_0^T|x'_k|^2\mathrm{d}t\leq\beta M+a_5+\|x_k\|.
\end{align}
By (V7), \eqref{4.28} and \eqref{4.29}, one has
\begin{align}
a_1\int_0^T|x_k^\bot|^\alpha\mathrm{d}t&\leq a_2T+\int_0^T V(x_k)\mathrm{d}t\\
&\leq a_2T+\frac{1}{2}\int_0^T|x'_k|^2\mathrm{d}t+M\\
&\leq a_6\|x_k\|+a_7,
\end{align}
for some constants $a_6,\ a_7>0$. Let
$$\bar{x}_k=\lim\limits_{t\rightarrow\infty}\frac{1}{t}\int_0^tx_k(s)\mathrm{d}s.$$
It is easy to see $\bar{x}_k=\frac{1}{T}\int_0^Tx_k^\bot(s)\mathrm{d}s,$ and there exist constants $a_8\ a_{9},\ a_{10}>0$ such that
\begin{align}\label{4.33}
|\bar{x}_k|^2&=\left|\frac{1}{T}\int_0^Tx_k^\bot(s)\mathrm{d}s\right|^2\leq\frac{1}{T^2}\left(\int_0^T|x_k^\bot(s)|\mathrm{d}s\right)^2\nonumber\\
&\leq a_{8}\left(\int_0^T|x_k^\bot(s)|^\alpha\mathrm{d}s\right)^{\frac{2}{\alpha}}\leq a_{9}+ a_{10}\|x_k\|^{\frac{2}{\alpha}}.
\end{align}
By Lemma \ref{lemma3.1}, \eqref{4.29} and \eqref{4.33} we have
$$\|x_k\|^2\leq a_{11}+a_{12}\|x_k\|+a_{13}\|x_k\|^{\frac{2}{\alpha}}$$
for some constants $a_{11},\ a_{12},\ a_{13}>0$. Since $\alpha\in(1,2)$, there exists a constant $a_{14}>0$ such that
$$\|x_k\|\leq a_{14}.$$

The remaining proof is the same as in Theorem  \ref{theorem1}, and we omit it.
\end{proof}

\section*{Acknowledgment}
This work was supported by National Basic Research Program of China (grant No. 2013CB834102), NSFC (grant No. 11571065, 11171132, 11201173), Science and Technology Developing Plan of Jilin Province (No. 20180101220JC), JilinDRC (No. 2017C028-1) and the Fundamental Research Funds for the Central Universities (No. 2412018QD036).

\baselineskip 9pt \renewcommand{\baselinestretch}{1.08}

\end{document}